\documentclass{amsart}
\usepackage[dvipsnames]{xcolor} %
\usepackage{amssymb,amsmath,amsfonts,amsthm, enumerate, tikz, float,enumitem, comment, mathtools, bm, xfrac, empheq, esvect, soul}
\usetikzlibrary{arrows.meta, positioning, fit, trees}
\usepackage[pdfencoding=auto, psdextra]{hyperref}
\mathtoolsset{showonlyrefs}
\usepackage{subcaption}

\usepackage{theoremref}
\newtheorem{theorem}{Theorem}

\newtheorem{lemma}[theorem]{Lemma}

\newcommand{\C}{C}
\newcommand{\bE}{\bar{\mathcal E}}
\newcommand{\Wce}{W_{\mathtt{CE}}}

\newcommand{\f}{\frac}

\renewcommand{\P}{\mathbf{P}}

\newcommand{\ind}[1]{\mathbf{1}{\{ #1 \}}}

\newcommand{\E}{\mathbf E}
\newcommand{\B}{\mathcal B}

\DeclareMathOperator{\Exp}{Exp}

\usepackage{lineno}

\newcommand{\ta}[1]{}%

\newcommand{\om}[1]{}%

\definecolor{darkgreen}{rgb}{0,0.5,0}

\usetikzlibrary[topaths]
\newcount\mycount

\title{Chase-escape with conversion on the complete graph}

\author[]{Matthew Junge} \address{Baruch College}  \email{Matthew.Junge@baruch.cuny.edu}

\author[]{Sergio Rodr\'{i}guez} \address{University of Puerto Rico, Rio Piedras Campus} \email{sergio.rodriguez20@upr.edu}

\begin{document}

\begin{abstract}
We prove that chase-escape with conversion on the complete graph undergoes a phase transition at equal fitness and derive simple asymptotic formulas for the extinction probability, the total number of converted sites, and the expected number of surviving sites.
\end{abstract}

\thanks{}

\maketitle

\section{Introduction}

\emph{Chase-escape with conversion}  takes place on a graph with vertices initially in either a red, blue, or white state. White vertices change to red according to independent rate-$\lambda$ Poisson processes along each edge leading to a red vertex. Red vertices change to blue according to independent rate-$1$ Poisson processes along each edge leading to a blue vertex. Additionally, each red vertex changes to blue according to an independent Poisson process with rate $\alpha \geq 0$. Blue is a terminal state.  

Bailey et al.\ introduced the model to describe central nervous system damage from multiple sclerosis \cite{bailey2025chase}. In their interpretation, white sites represent healthy neurons, red sites have inflammatory activity, and blue sites regulatory activity. These dynamics generalize the \emph{chase-escape} process, in which $\alpha =0$, introduced by ecologists in the 1990s to model parasite-host relations \cite{rand1995invasion}. Chase-escape is challenging to analyze and admits many interpretations, including predator-prey interactions, rumor scotching, infection spread, and malware repair. 
See \cite{kordzakhi2006two,rumor, cande, kumar2021chase, hinsen2020phase, beckman2021chase} for overviews.

The \emph{standard initial conditions} for chase-escape with conversion consist of a red root vertex and all other sites white. The model \emph{fixates} at the moment no red sites remain. 
Bailey et al.\ 
studied monotonicity properties for \emph{total damage} i.e., the number of remaining blue sites when the process fixates. Among other findings, they proved that the total damage is monotone in the parameters for the process on the complete graph on $n$ vertices, which we denote by $K_n$ \cite[Theorem 1]{bailey2025chase}.

Kortchemski conducted a detailed analysis of chase-escape on $K_{n+2}$ with an initial condition of one red, one blue, and $n$ white sites \cite{complete}. Let $\Wce= \Wce(K_{n+2},\lambda)$ be the number of remaining white sites when the process fixates. We say that \emph{extinction} occurs when $\Wce = 0$. In \cite[Theorem 1]{complete}, he proved that
\begin{align} \label{eq:KW}
\lim_{n \to \infty} \P(\Wce(K_{n+2},\lambda) =0 )  = \begin{cases} 0, &\lambda <1 \\ 1/2, & \lambda =1 \\ 1, & \lambda >1 \end{cases}. 
\end{align}
Thus, the extinction event undergoes a phase-transition at ``equal fitness" when both red and blue spread at rate $1$.

It is natural to ask what impact conversion has on the phase-transition. Returning to the standard initial conditions, let $W=W(K_{n+1},\lambda,\alpha)$ be the number of remaining white sites when chase-escape with conversion fixates. The three theorems below hold for all $\alpha >0$ and the limits are as $n \to \infty$. 
\begin{theorem}\thlabel{thm:W} %
$\displaystyle \P(W(K_{n+1}, \lambda, \alpha)=0 )  \to \begin{cases} 0, &\lambda <1 \\ 1/2^\alpha, & \lambda =1 \\ 1, & \lambda >1 \end{cases}$. %

\end{theorem}

Hence $\alpha$ changes the extinction probability but not the location of the phase transition. The intuitive reason for this is as follows. At first, chase and conversion occur infrequently because red has many white vertices to escape to. As the number of white sites begins to dwindle, red growth slows then reverses. Chase then becomes the dominant force, since its effect is dictated by the  product of the red and blue populations. All the while, conversion is linear in the red population size and, thus, never the main driver. See Figure~\ref{fig:sim} for a sample realization. To quantify its effect, we prove that the number of converted sites scales like $\alpha \log n$. Let $\C=\C(K_{n+1}, \lambda ,\alpha)$ be the number of blue sites at fixation that turned blue from conversion.

\begin{theorem}\thlabel{thm:C}
    $\C(K_{n+1}, 1, \alpha)/\log n \to \alpha$ in probability. %
\end{theorem}

\begin{figure}
    \centering
    \includegraphics[width=0.9\linewidth]{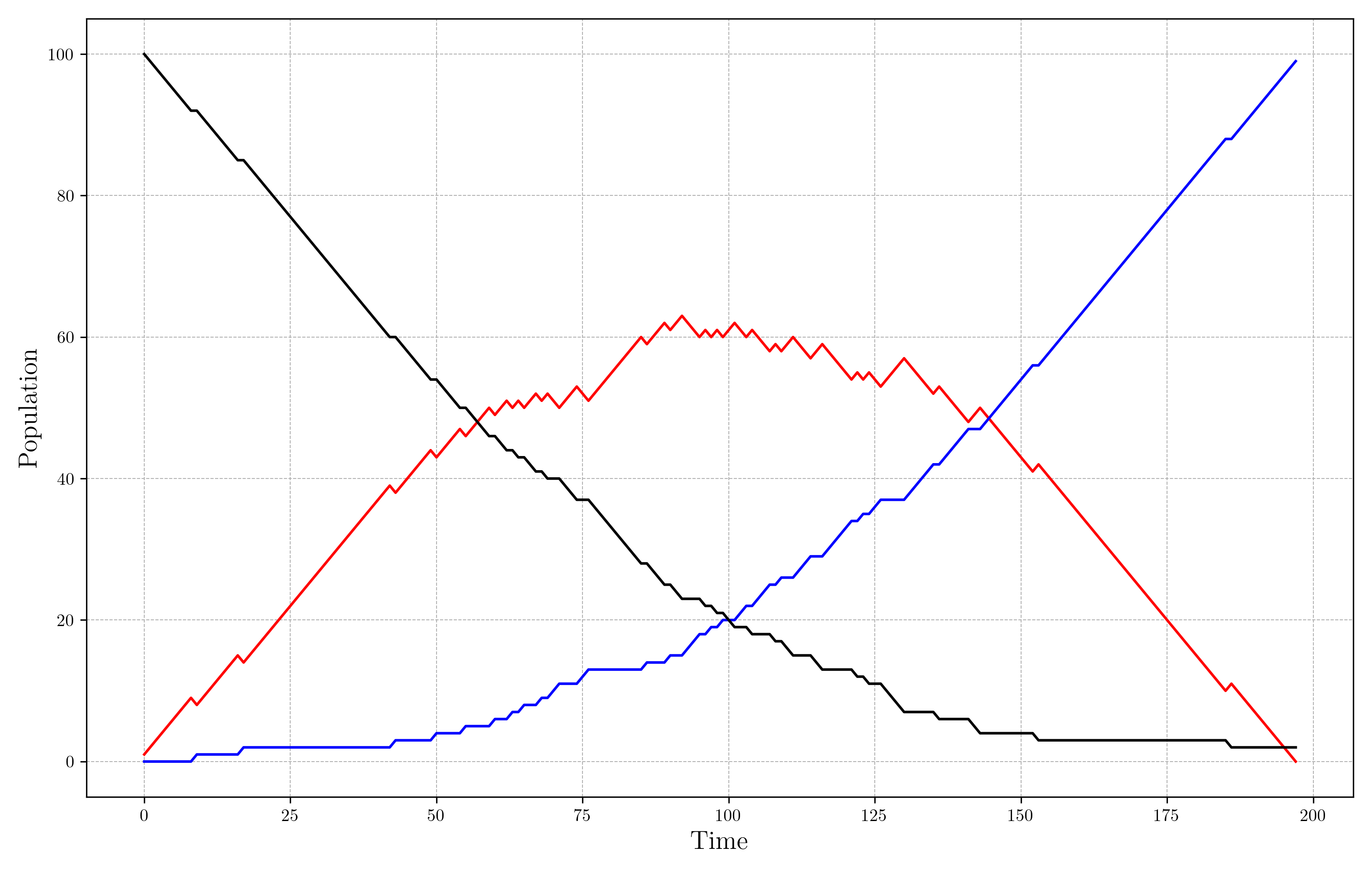}
    \caption{Red, white, and blue population sizes at the jump times in chase-escape with conversion on $K_{101}$ with $\lambda = 1$ and $\alpha=4$. In this simulation $W(K_{101},1,4) = 2$.}
    \label{fig:sim}
\end{figure}

Regarding the expected value, \cite[Theorem 5]{complete} proved that $\E[\Wce(K_{n+2},1)] \to 2.$ We extend this to our setting.

\begin{theorem}\thlabel{thm:EW}
   $\E[W(K_{n+1}, 1, \alpha)] \to 2\alpha$. %
\end{theorem}

In all three theorems $\alpha$ has a subtle, yet surprisingly simple effect at criticality.
The reason, which was missed in \cite[Remark 10]{bailey2025chase}, is that we may view chase-escape with conversion on $K_{n+1}$ as an inhomogeneous version of chase-escape on $K_{n+2}$ for which the initially blue site chases red at rate $\alpha$ rather than rate $1$. Sites that turn blue later on chase at rate $1$. This observation explains why the formulas in \thref{thm:W} and \thref{thm:EW} agree when $\alpha =1$, since the two processes are in fact equivalent. 

Our theoretical contribution is making the aforementioned observation and then fitting it into Kortchemski's machinery to find the deeper reason for the simple formulas. Namely, a key terminal value, that was a unit exponential random variable in \cite{complete}, now follows a $\text{Gamma}(\alpha,1)$ distribution.  
Note that \cite{complete} contains more results than presented here. Most should be possible to extend to include conversion. We chose to focus on what we thought were conversion's most notable consequences.

\section{Overview}
Section~\ref{sec:setup} begins by describing the evolution of the different populations in chase-escape then explains why these quantities are equivalent to the counts in independent death and birth processes. Differing from \cite{complete}, the birth process we consider begins with a single defective individual who has birth rate $\alpha$ (\cite{bailey2025chase} interpreted this as constant immigration). We show in \thref{lem:tvW} and \thref{lem:tvB} that both processes have terminal values $\bE\sim \Exp(1)$ and $G_\alpha\sim \text{Gamma}(\alpha,1)$ that determine the limiting growth rates. Section~\ref{sec:proofs} explains why the arguments in \cite{complete} continue to apply but with comparisons between $\bE$ and $G_\alpha$. The appendix has two integral calculations needed for the comparisons. 

\section{Setup} \label{sec:setup}

 \subsection{Chase-escape with conversion}
 We describe the population dynamics of chase-escape with conversion on $K_{n+1}$. Let $(R_t)_{t \in \mathbb N}$ be the number of vertices in state $r$ at time $t$. Note that $R_0 = 1$. Similarly, let $(B_t)_{t \in \mathbb N}$ be the number of blue vertices with $B_0 = 0$. Let $W_t = n+1 - R_t - B_t$ be the number of vertices in state $w$. Conditional on $R_t = r$, $B_t=b$, and $W_t = w = n+1 - (r +b)$, we have the following transition probabilities at the next jump time:
\begin{align}\label{eq:p}
\f{ \lambda w r}{\lambda r w + b r + \alpha r} \quad \text{ and } \quad   \f{br +\alpha r}{\lambda r w + b r + \alpha r},
\end{align}
 where the first is the probability the red population increases by 1 and the second is the probability the red population decreases by one while the blue population increases by 1. Notice that the $r$ factor is common to all terms and can be canceled. %

\subsection{Birth and death processes}

Following \cite{complete, bailey2025chase}, we now define birth and death processes with the same transition probabilities as chase-escape with conversion.
The \emph{death process} starts with $n$ individuals, each of which die at exponential rate $\lambda$. Let $\mathcal W_t$ denote the number of individuals remaining at time $t$. Let $\delta(i)$ denote the time (in continuous units) of the $i$th death. The \emph{defective birth process} starts with 1 individual who generates an additional individual at exponential rate $\alpha$. All future offspring generate additional children at exponential rate 1. Let $\mathcal B_t$ denote the number of individuals at time $t$. Let $\beta(i)$ denote the time of the $i$th jump in the birth process. 
By construction,  $\delta(i+1) - \delta(i) \sim \Exp( \lambda (n-i))$ for $0 \leq i \leq n-1$ and $\beta(i+1) - \beta(i) \sim \Exp(i + \alpha)$ for $i \geq 0$.

As the times between jumps are exponential, it is easy to verify that $\mathcal W_t$ and $\mathcal B_t$ have the same transition probabilities \eqref{eq:p} as $B_t$ and $W_t$ at the jump times. Hence the two processes can be coupled at their jump times up until no white or red sites remain.

\subsection{Terminal values}

As discussed in \cite[Section 1.2]{complete}, the reversal $\overline {\mathcal W}_t := \mathcal W_{\delta(n) - t}$ gives a pure birth process with birth rate $\lambda$ stopped just before its $n$th jump time. As $n\to \infty$,  $\overline {\mathcal W}_t$ converges to the pure birth process  $\overline{\mathcal Y}_t$ with rate $\lambda$ that starts with one individual. The classical result \cite[Theorem I]{kendall1966branching} states that $\overline{\mathcal Y}_t$ has a limiting \emph{terminal value} that dictates its growth. We record this as a lemma. 
\begin{lemma}\thlabel{lem:tvW}
    $e^{-\lambda t} \overline {\mathcal Y}_t \overset{a.s.}{\to} \bE$ with $\bE \sim \Exp(1)$.
\end{lemma}

As for $\B_t$, when $\alpha\in \mathbb N$, we may alternatively view the defective progenitor as $\alpha$ distinct progenitors who each give birth at rate 1. Thus, we may write $\B_t = \B_t'-\alpha +1$ with $\B_t'$ a \emph{pure birth process} (all individuals give birth at rate $1$) with initial condition $\B_0' = \alpha$. Kendall also studied this initial condition in \cite[Theorem I]{kendall1966branching} proving that $e^{-t} \mathcal B_t'$ converges almost surely to a {terminal value} $G_\alpha\sim\text{Gamma}(\alpha,1)$ i.e.\ has density function $e^{-x}x^{\alpha -1}/\Gamma(\alpha)$ for $x > 0$ \cite{kendall1966branching}.  We do not doubt that Kendall could have handled the non-integer case. Nonetheless, we failed to find a reference, so we explain why it follows from the pure birth case.

\begin{lemma}\thlabel{lem:tvB}
    For any $\alpha >0$ it holds that $e^{-t} \mathcal B_t \overset{a.s.} \to G_\alpha$ with $G_\alpha \sim \text{\emph{Gamma}}(\alpha,1)$.
\end{lemma}

\begin{proof}
     For notational convenience, we will write $\mathcal B_\alpha(t) = \mathcal B_t$. Our starting point is to write 
     $$\B_{\alpha}(t) = \textstyle \sum_{i=1}^{N(t)}\B_i(t-T_i)$$ where $N(t)$ is a Poisson point process with rate $\alpha$, the $T_i$ are its jump times, and the $\B_i$ are independent pure birth processes each starting with one individual. Scaling by $e^{-t}$ and applying \cite[Theorem I]{kendall1966branching} gives
    \begin{align*}
        e^{-t}\B_{\alpha}(t) &= e^{- t}\sum_{i=1}^{N(t)}\B_i(t-T_i) = \sum_{i=1}^{N(t)}  e^{- T_i}e^{- (t-T_i)}\B_i(t-T_i) \overset{a.s.}\to \sum_{i=1}^\infty e^{- T_i} \mathcal E_i,
    \end{align*}
    with the $\mathcal E_i$ independent unit exponential random variables.
    
    It remains to prove that the limiting random sum  is distributed as $G_\alpha$. 
    Let $X:= \sum_{i=1}^\infty e^{- T_i} \mathcal E_i$ and $M_X(s) := \E[e^{-sX}]$ be the Laplace transform. Some algebra gives
$$M_X(s) = \E \prod_{i=1}^\infty\f{e^{ T_i}}{e^{ T_i} + s} .$$
 As $(T_i)$ is a homogeneous Poisson point process with intensity $\alpha$ and the above is its probability generating functional evalauted at the function $e^{x}/(e^x +s)$, we conclude that
    $$M_X(s) = \exp\left(-\alpha\int_0^\infty 1-\f{e^{ x}}{e^{ x} + s} \: dx\right)=(1+s)^{-\alpha}.$$
    The function $(1+s)^{-\alpha}$ is the Laplace transform of $G_\alpha$, which implies the claimed equivalence. 
\end{proof}

\section{Theorem Proofs} \label{sec:proofs}

\subsection{Proof of \thref{thm:W}}
The argument for $\lambda \neq 1$ is the same as the proof of \cite[Theorem 1]{complete}. 
This is because conversion is irrelevant when it comes to extinction since the birth and death processes have altogether different exponential growth rates. 
When $\lambda=1$ the same reasoning as \cite[Theorem 1]{complete} gives that whether or not extinction occurs comes down to comparing terminal values, so by \thref{lem:tvW} and \thref{lem:tvB} we have
$$\lim_{n\to \infty} \P(W(K_{n+1}, 1,\alpha)=0) = \P(G_\alpha < \bE).$$
\thref{lem:G<} shows that $\P(G_\alpha < \bE) = 2^{-\alpha}$.

\subsection{Proof of \thref{thm:C}}

We seek to prove that
\begin{align}\label{eq:lln}
\lim_{n \to \infty} \P( |\C(K_{n+1}, 1, \alpha) - \alpha \log n| > \epsilon \log n) = 0 \text{ for all $\epsilon >0$.}
\end{align}
Recall that $\delta(i)$ are the times of the $i$th death and $\beta(i)$ are the times of the $i$th birth in the death and defective birth processes from Section~\ref{sec:setup}. Let $$\tau = \tau(K_{n+1}, 1, \alpha) := \beta( n \wedge \min \{ i \geq 1 \colon \beta(i) < \delta(i) \})$$
be the time of whichever comes first between the $n$th birth and/or when the birth process first surpasses the death process. A moment's thought reveals that $\tau$ is precisely the time of fixation. Hence,  $\C$ can be coupled with the number of births by the defective progenitor up to time $\tau$. 

Combining \cite[Equation (11) and Lemma 3.1]{complete}, which continue to hold with conversion, yields that
\begin{align} \label{eq:tau}
    \tau(K_{n+1}, 1, \alpha) / \log n \to 1 \text{ in probability as $ n \to \infty$.}
\end{align} 
 As the number of births by the defective progenitor is a Poisson point process with intensity $\alpha$, we obtain \eqref{eq:lln} using standard concentration estimates  and \eqref{eq:tau}.

\subsection{Proof of \thref{thm:EW}}
The event that the initial red site is instantly converted has a non-vanishing contribution to the expected value. Specifically, 
$$\E[W \mid W=n]\P(W=n) = n \f{\alpha}{n+\alpha} \to \alpha.$$
Conditioning on the above event and its complement, then following the proof of \cite[Theorem 5 (ii)]{complete} gives
$$\lim_{n \to \infty} \E[W(K_{n+1},1,\alpha)] = \alpha + \E[Z],$$
where $Z \sim [1+\text{Poisson}((G_\alpha - \bE))]\ind{G_\alpha > \bE}$. As the expected value of a Poisson distribution with random parameter is the expected value of the random parameter, we have
$$\E[Z] = \P(G_\alpha > \bE) + \E[ (G_\alpha - \bE) \ind{G_\alpha > \bE}].$$
The calculation in \thref{lem:EG} shows that
$$\E[ (G_\alpha - \bE) \ind{G_\alpha > \bE}] = \alpha - (1 - 2^{-\alpha}).$$ Since $\P(G_\alpha > \bE)= 1-2^{-\alpha}$ by \thref{lem:G<}, we have $\E[Z] = \alpha$, completing the proof. 

\appendix

\section{Integral calculations}
Let $G_\alpha \sim \text{Gamma}(\alpha,1)$ and $\bE \sim \Exp(1)$ be independent random variables. 
\begin{lemma} \thlabel{lem:G<}
    $\P(G_\gamma < \bE) = 2^{-\alpha}$.
\end{lemma}

\begin{proof}
We have
\begin{align}
\mathbf{P}(G_\alpha<\overline{\mathcal{E}}) =\int_0^\infty \mathbf{P}(\overline{\mathcal{E}}>x)\,f_{G_\alpha}(x)\,dx
&=\int_0^\infty e^{-x}\,\frac{x^{\alpha-1}e^{-x}}{\Gamma(\alpha)}\,dx\\
&=\frac{1}{\Gamma(\alpha)}\int_0^\infty x^{\alpha-1}e^{-2x}\,dx.
\end{align}
Using $\int_0^\infty x^{\alpha-1}e^{-\beta x}\,dx=\Gamma(\alpha)\beta^{-\alpha}$ for $\alpha,\beta>0$,
\[
\mathbf{P}(G_\alpha<\overline{\mathcal{E}})=\frac{\Gamma(\alpha)2^{-\alpha}}{\Gamma(\alpha)}=2^{-\alpha}.
\]

\end{proof}

\begin{lemma} \thlabel{lem:EG}
    $\E[ (G_\alpha - \bE) \ind{G_\alpha > \bE}] = \alpha - (1 - 2^{-\alpha})$
\end{lemma}

\begin{proof}
We have
\begin{align}
\mathbf{E}\!\left[(G_\alpha-\overline{\mathcal{E}})\mathbf{1}_{\{G_\alpha>\overline{\mathcal{E}}\}}\right]
&=\iint_{x>y} (x-y)\, f_{G_\alpha}(x)f_{\overline{\mathcal{E}}}(y)\,dy\,dx\\
&=\int_0^\infty\!\!\left(\int_0^x (x-y)e^{-y}\,dy\right) f_{G_\alpha}(x)\,dx.
\end{align}
The inner integral is
\begin{align}
\int_0^x (x-y)e^{-y}\,dy
= x(1-e^{-x})-\!\int_0^x y e^{-y}\,dy
&= x(1-e^{-x})-\big(1-(x+1)e^{-x}\big)\\
&= x-1+e^{-x}.
\end{align}
Hence
\begin{align}
\mathbf{E}\!\left[(G_\alpha-\overline{\mathcal{E}})\mathbf{1}_{\{G_\alpha>\overline{\mathcal{E}}\}}\right]
&= {\mathbf{E}[G_\alpha]}\;-\;1\;+\;\mathbf{E}[e^{-G_\alpha}].
\end{align}
We obtain the claimed formula after noting that $\E[G_\alpha] =\alpha$ and $\mathbf{E}[e^{-G_\alpha}]= 2^{-\alpha}$.
The first equality is a basic fact about the Gamma distribution and the latter uses the Laplace transform $\mathbf{E}[e^{-sG_\alpha}] = (1+s)^{-\alpha}$ at $s=1$.
\end{proof}

\bibliographystyle{amsalpha}

\bibliography{paper.bib}

\newcommand{\etalchar}[1]{$^{#1}$}
\providecommand{\bysame}{\leavevmode\hbox to3em{\hrulefill}\thinspace}
\providecommand{\MR}{\relax\ifhmode\unskip\space\fi MR }
\providecommand{\MRhref}[2]{%
  \href{http://www.ams.org/mathscinet-getitem?mr=#1}{#2}
}
\providecommand{\href}[2]{#2}
\begin{thebibliography}{BBHT{\etalchar{+}}25}

\bibitem[BBHT{\etalchar{+}}25]{bailey2025chase}
Emma Bailey, Erin Beckman, Sara{\'\i} Hern{\'a}ndez-Torres, Matthew Junge, Aanjaneya Kumar, Ann Lee, Danny Li, Alisher Raufov, Lily Reeves, and Omer Rondel, \emph{Chase-escape with conversion as a multiple sclerosis lesion model}, arXiv:2507.21235 (2025).

\bibitem[BCE{\etalchar{+}}21]{beckman2021chase}
Erin Beckman, Keisha Cook, Nicole Eikmeier, Sara\'{i} Hern\'{a}ndez-Torres, and Matthew Junge, \emph{Chase-escape with death on trees}, The Annals of Probability \textbf{49} (2021), no.~5, 2530--2547.

\bibitem[Bor08]{rumor}
Charles Bordenave, \emph{On the birth-and-assassination process, with an application to scotching a rumor in a network}, Electronic Journal of Probability \textbf{13} (2008), 2014--2030.

\bibitem[DJT20]{cande}
Rick Durrett, Matthew Junge, and Si~Tang, \emph{Coexistence in chase-escape}, Electronic Communications in Probability \textbf{25} (2020), 1--14.

\bibitem[HJCW20]{hinsen2020phase}
Alexander Hinsen, Benedikt Jahnel, Elie Cali, and Jean-Philippe Wary, \emph{{Phase transitions for chase-escape models on Poisson–Gilbert graphs}}, Electronic Communications in Probability \textbf{25} (2020).

\bibitem[Ken66]{kendall1966branching}
David~G Kendall, \emph{Branching processes since 1873}, Journal of the London Mathematical Society \textbf{1} (1966), no.~1, 385--406.

\bibitem[KGD21]{kumar2021chase}
Aanjaneya Kumar, Peter Grassberger, and Deepak Dhar, \emph{Chase-escape percolation on the 2d square lattice}, Physica A: Statistical Mechanics and its Applications (2021), 126072.

\bibitem[KL06]{kordzakhi2006two}
George Kordzakhia and Steven Lalley, \emph{A two-species competition model on $\mathbb{Z}^d$}, Stochastic Processes and their Applications \textbf{115} (2006).

\bibitem[Kor15]{complete}
Igor Kortchemski, \emph{A predator-prey {SIR} type dynamics on large complete graphs with three phase transitions}, Stochastic Processes and their Applications \textbf{125} (2015), no.~3, 886 -- 917.

\bibitem[RKW95]{rand1995invasion}
DA~Rand, Matthew Keeling, and HB~Wilson, \emph{Invasion, stability and evolution to criticality in spatially extended, artificial host—pathogen ecologies}, Proceedings of the Royal Society of London. Series B: Biological Sciences \textbf{259} (1995), no.~1354, 55--63.

\end{thebibliography}

\end{document}